\documentclass[a4paper]{amsart}

\usepackage{packages}
\usepackage{style}
\usepackage{notations}

\begin{document}

\title{An arithmetic Zariski 4-tuple of twelve lines}
\author{Benoît Guerville-Ballé}
\address{
Institut Joseph Fourier,\newline\indent
UMR 5582 CNRS-UJF, \newline\indent
Université Grenoble Alpes \newline\indent
38 000 Grenoble, France\newline\indent
}
\email{benoit.guerville-balle@math.cnrs.fr}

\thanks{Partially supported by JSPS-MAE Sakura Program}				% Pour les financements

		% Code AMS

\begin{abstract}
	Using the invariant developed in~\cite{ArtFloGue}, we differentiate four arrangements with the same combinatorial information but in different deformation classes. From these arrangements, we construct four other arrangements such that there is no orientation-preserving homeomorphism between them. Furthermore, some couples of arrangements among this 4-tuplet form new arithmetic \textsc{Zariski} pairs, \emph{i.e.} a couple of arrangements conjugate in a number field with the same combinatorial information but with different embedding topology in $\CC\PP^2$.
\end{abstract}

\maketitle

\section{Introduction}

The study of the relation between topology and algebraic geometry was initiated by F.~\textsc{Klein} and H.~\textsc{Poincar\'e} at the beginning of the XX century. It was known by the work of O.~\textsc{Zariski}~\cite{Zar_Problem,Zar_Irregularity,Zar_Poincare} that the topology of the embedding of an algebraic curve $C$ in the complex projective plane $\CC\PP^2$ is not determined by the local topology of its singular points. Indeed, he used the fundamental group of their \emph{complement} $E_{C_i}=\CC\PP^2\setminus C_i$  as a strong topological invariant to prove that two sextics $C_1,C_2$ have the same \emph{combinatorial information} but different \emph{topological types}. Such couple of curves was called a \textsc{Zariski} pair by E.~\textsc{Artal}~\cite{Art_Couples}. There are many examples of \textsc{Zariski} pairs (or larger $k$-tuplets) that have been discovered (see for examples~\cite{ArtCogTok,CasEyrOka_Topology,Deg_Deformations,Oka_Transforms,Shimada_Fundamental}), while essentially two \textsc{Zariski} pairs of \emph{line arrangements} have been known. The first is a couple of arrangements with complex equations constructed by G.~\textsc{Rybnikov} \cite{Ryb_Preprint,Ryb_Article,ArtCarCogMar_Rybnikov}, and the second is a pair of complexified real arrangements obtained by E.~\textsc{Artal}, J.~\textsc{Carmona}, J.I.~\textsc{Cogolludo} and M.A.~\textsc{Marco}~\cite{ArtCarCogMar_Topology}. The proof of the former is done using the lower central series of the fundamental group, and the latter using the braid monodromy. Thus, in both, the use of a computer were mandatory. This small number of examples for line arrangements, and the routine use of a computer show the difficulty to sense what characterizes the topological type of an arrangement.

In the present paper, new counterexamples to the combinatoriality of the topological type of arrangements are explicitly constructed. These arrangements are defined in the 10\up{th} cyclotomic field, and their equations are connected by the action of an element of the \textsc{Galois} group of $\QQ(\zeta_{10})$. These new pairs are arithmetic \textsc{Zariski} pairs, in particular their fundamental groups have the same profinite completion (\emph{i.e.} the same finite quotients). In contrast with~\cite{Ryb_Preprint}, and as in~\cite{ArtCarCogMar_Topology}, the major part of the computations needed to single out these arrangements are doable by hand.

In~\cite{ArtFloGue}, the authors construct a new topological invariant of line arrangements $\mathcal{I}(\A,\xi,\gamma)$ based on the inclusion map of the \emph{boundary manifold} (that is the boundary of a regular neighbourhood of the arrangement) in the complement. This invariant depends on a character of the fundamental group of the complement and a special cycle in the incidence graph of the arrangement; and it can be computed directly from the \emph{wiring diagram} of the arrangement. In this paper, we use this invariant to single out four arrangements (two couples of complex conjugate arrangements) with the same combinatorial structure but lying in different deformation classes (\emph{i.e.} an oriented and ordered \textsc{Zariski} $4$-tuplet). Combinatorially, they contain eleven lines, four points of multiplicity four, six triple points and some double points. Then, to delete all automorphisms of the combinatorics, we add a twelfth line to these arrangements (as in~\cite{ArtCarCogMar_Topology}). Thus we construct new \textsc{Zariski} pairs. 

Currently, we do not know if they have isomorphic fundamental groups. However, the invariant $\mathcal{I}(\A,\xi,\gamma)$ allows us to compute the quasi-projective part of the characteristic varieties; more precisely to determine the quasi-projective depth of the character $\xi$ (see~\cite{Art_Singularities, Gue_Thesis}). Unfortunately in the present case they are equal.  Furthermore, the combinatorial structure of the arrangements satisfies the hypotheses of \cite{DimIbaMac}, so the projective part of the characteristic varieties is combinatorially determined. Thus the question of the combinatoriality of the characteristic varieties is still open.

In the Section~\ref{Section_Arrangements}, we give usual definitions and define arrangements $\mathcal{N}^+$, $\mathcal{N}^-$, $\mathcal{M}^+$ and $\mathcal{M}^-$ forming the oriented and ordered \textsc{Zariski} $4$-tuplet.  In Section~\ref{Section_Pairs}, we apply a classical argument to the arrangements $\mathcal{N}^+$, $\mathcal{N}^-$, $\mathcal{M}^+$ and $\mathcal{M}^-$ to construct the new examples of arithmetic \textsc{Zariski} pairs. The last section is divided into two parts. In the first one, we recall the construction and the definition of the invariant $\mathcal{I}(\A,\xi,\gamma)$; in the second one, we give the wiring diagrams of $\mathcal{N}^+$ and $\mathcal{M}^+$ required to compute the invariant together with the character $\xi$ and the cycle $\gamma$ allowing us to distinguish them. Then we compute the invariant for the four arrangements, thus we prove that $(\mathcal{N}^+,\mathcal{N}^-,\mathcal{M}^+,\mathcal{M}^-)$ forms an oriented and ordered \textsc{Zariski} $4$-tuplet.

%%%%%%%%%%%%%%%%%%%%%%%%%%%%%%%%%%%%%%%%%%%%%%
%%%%%%%%%%%%%%%%%%%%%%%%%%%%%%%%%%%%%%%%%%%%%%
\section{The arrangements}\label{Section_Arrangements}
%%%%%%%%%%%%%%%%%%%%%%%%%%%%%%%%%%%%%%%%%%%%%%
%%%%%%%%%%%%%%%%%%%%%%%%%%%%%%%%%%%%%%%%%%%%%%

In this section, a brief recall on combinatorics and realization is given (see~\cite{OrlTer} for definitions of classical objects) , together with the description of the arrangements allowing to construct the new example of \textsc{Zariski} pairs.

%%%%%%%%%%%%%%%%%%%%%%%%%%%%%%%%%%%%%%%%%%%%%%
\subsection{The combinatorics}\label{Subsection_Combinatorics}\mbox{}
%%%%%%%%%%%%%%%%%%%%%%%%%%%%%%%%%%%%%%%%%%%%%%

%%%%%%%% Définition des objets combinatoires
\begin{de}
	A \emph{combinatorial type} (or simply a combinatorics) is a couple $\C=(\L,\P)$, where $\L$ is a finite set and $\P $ a subset of the power set of $\L$, satisfying that:
	\begin{itemize}
		\item For all $P \in \P,\ \sharp P \geq 2$;
		\item For any $L_1,L_2 \in \L,\ L_1\neq L_2,\ \exists ! P\in \P$ such that $L_1,L_2\in P$.
	\end{itemize}
An \emph{ordered combinatorics} $\C$ is a combinatorics where $\L$ is an ordered set.
\end{de}

\begin{de}
	Let $\C=(\L,\P)$ be a combinatorics. An \emph{automorphism} of $\C$ is a permutation of $\L$ preserving~$\P$. The set of such automorphisms is the \emph{automorphism group} of  the combinatorics $\C$.
\end{de}

The combinatorics can be encoded in the incidence graph, which is a subgraph of the Hasse diagram:

\begin{de}\label{Definition_IncidenceGraph}
	 The \emph{incidence graph} $\Gamma_{\C}$ of a combinatorics $\C=(\L,\P)$ is a non-oriented bipartite graph where the set of vertices $V(\C)$ is decomposed in two disjoint sets:
	\begin{equation*}
		V_\P(\C)=\{v_P\mid P\in \P\}\quad \text{and}\quad
		V_\L(\C)=\{v_L\mid L\in \L\}.
	\end{equation*}
	An edge of $\Gamma_{\C}$ joins $v_L\in V_\L(\C)$ to $v_P\in V_\P(\C)$ if and only if $L\in P$. 
\end{de}

\begin{rmk}
	The automorphism group of $\C$ is isomorphic to the group of automorphism of $\Gamma_\C$ respecting the structure of bipartite graph (that is sending $V_\P(\C)$ (resp. $V_\L(\C)$) on itself). Generally it is smaller than the automorphism group of the graph.
\end{rmk}

%%%%%%  Definition de la combinatoire qui nous intéresse
The starting point to construct and to detect the new example of \textsc{Zariski} pairs is the combinatorics $\mathcal{K}=(\L,\P)$ (obtained from a study of the combinatorics with 11 lines) defined by $\L=\set{L_1,\cdots,L_{11}}$ and 
\begin{multline*}
\P=\left\{\set{1, 2}, \set{1, 3, 5, 7}, \set{1, 4, 6, 8}, \set{1, 9}, \set{1, 10, 11}, \set{2, 3, 6, 9}, \set{2, 4, 5, 10}\right. , \\
\set{2, 7, 11}, \set{2, 8}, \set{3, 4}, \set{3, 8, 11}, \set{3, 10}, \set{4, 7}, \set{4, 9, 11}, \set{5, 6}, \\
\left. \set{5, 8, 9}, \set{5, 11}, \set{6, 7, 10}, \set{6, 11}, \set{7, 8}, \set{7, 9}, \set{8, 10}, \set{9, 10} \right\}.
\end{multline*}

\begin{propo}
	The automorphism group of $\mathcal{K}$ is cyclic of order 4, and is generated by:
	\begin{equation*}
		\sigma=(1\ 3\ 2\ 4)(5\ 6)(7\ 9\ 10\ 8).
	\end{equation*} 
\end{propo}

\begin{proof}
	Let $\phi$ be an automorphism of $\mathcal{K}$. The line $L_{11}$ is the only one containing 4 triple points, thus it is fixed by the $\phi$. Since $L_5$ and $L_6$ are the only one intersecting $L_{11}$ in a double point, then they are in distinct $\phi$-orbits than the other lines. By the same way, the lines $L_1$, $L_2$, $L_3$ and $L_4$ contain 2 double points, one triple point and 2 quadruple points, thus they are in $\phi$-orbits distinct of the other lines. The same argument work for the lines $L_7$, $L_8$, $L_9$ and $L_{10}$. Thus the decomposition of $\set{1,\cdots,11}$ in $\phi$-orbits is a sub-decomposition of:
	\begin{equation}\label{decompo}\tag{I}
		\set{1,2,3,4} \sqcup \set{5,6} \sqcup \set{ 7,8,9,10} \sqcup \set{11}.
	\end{equation}
	We decompose the following in two parts.
	
	First, we assume that $\phi(5)=6$ and $\phi(6)=5$. Then $\sigma(\set{5,8,9})=\set{6,7,10}$ and thus $\phi(\set{8,9})=\set{7,10}$. 
	\begin{enumerate}
	
		\item If $\phi(8)=7$ and $\phi(9)=10$ then $\phi(\set{3,8,11})=\set{2,7,11}$ and thus $\phi(3)=2$. Using this and decomposition~\ref{decompo}, we have $\phi(\set{3,10})=\set{2,8}$ and then $\phi(10)=8$. To finish, using the sets $\set{1,9}$, $\set{2,8}$, $\set{3,10}$ and $\set{4,7}$ and decomposition~\ref{decompo}, we obtain that $\phi=(1\ 3\ 2\ 4)(5\ 6)(7\ 9\ 10\ 8)$.
		
		\item If $\phi(8)=10$ and $\phi(9)=7$ then by the same way we obtain that $\phi=(4\ 2\ 3\ 1)(5\ 6)(8\ 10\ 9\ 7)$. 
		
	\end{enumerate}
	
	Second, we assume that $\phi(5)=5$ and $\phi(6)=6$. Then the 4 quadruple points and decomposition~\ref{decompo} imply that $\phi(\set{7,10})=\set{7,10}$ and $\phi(\set{8,9})=\set{8,9}$.
	\begin{enumerate}
	
		\item If $\phi_{\mid\set{7,8,9,10}}=\Id$ then the sets $\set{1,9}$, $\set{2,8}$, $\set{3,10}$, $\set{4,7}$ and decomposition~\ref{decompo} imply that $\phi=\Id$.
		
		\item If $\phi_{\mid\set{7,10}}=\Id$, $\phi(8)=9$ and $\phi(9)=8$ then the sets $\set{1,9}$, $\set{2,8}$, $\set{3,10}$ and $\set{4,7}$ and decomposition~\ref{decompo} imply that $\phi(2)=1$. Thus we have $\phi(\set{2,7,11})=\set{1,10,11}$ and then $\phi(7)=10$, which is impossible.
		
		\item If $\phi_{\mid\set{8,9}}=\Id$, $\phi(7)=10$ and $\phi(10)=7$ then by the same way we also obtain a contradiction. 
		
		\item If $\phi(7)=10$, $\phi(8)=9$, $\phi(9)=8$ and $\phi(10)=7$ then the sets $\set{1,9}$, $\set{2,8}$, $\set{3,10}$, $\set{4,7}$ and decomposition~\ref{decompo} imply that $\phi=(1\ 2)(3\ 4)(7\ 10)(8\ 9)$.
	
	\end{enumerate}
	We obtain that the automorphism group of $\mathcal{K}$ is the cyclic group generated by $\sigma=(1\ 3\ 2\ 4)(5\ 6)(7\ 9\ 10\ 8)$.
	
\end{proof}

\begin{rmk}
	The set $\P$ of $\mathcal{K}$ is decomposed in 8 orbits by the action of its automorphism group:
	\begin{itemize}
		\item[-] the four points of multiplicity 4,
		\item[-] the four points of multiplicity 3 of $L_{11}$,
		\item[-] the two other points of multiplicity 3,
		\item[-] two orbits with 2 double points,
		\item[-] two orbits with 4 double points,
		\item[-] a single isolated orbit composed of the intersection point of $L_5$ and $L_6$,
	\end{itemize}
	and the set $\L$ of $\mathcal{K}$ is decomposed in 4 orbits : $\set{L_1,L_2,L_3,L_4}\sqcup\set{L_5,L_6}\sqcup\set{L_7,L_8,L_9,L_{10}}\sqcup\set{L_{11}}$.
\end{rmk}

%%%%%%%%%%%%%%%%%%%%%%%%%%%%%%%%%%%%%%%%%%%%%%
\subsection{Complex realizations}\label{Subsection_Realization}\mbox{}
%%%%%%%%%%%%%%%%%%%%%%%%%%%%%%%%%%%%%%%%%%%%%%

The following definitions are given on the field of the complex numbers. 

%%%%%%% Définition des réalisations

\begin{de}
	Let $\A$ be a line arrangement. The combinatorics of $\A$ is the poset of all the intersection of the elements in $\A$, with respect to reverse inclusion.
\end{de}

\begin{rmk}
	The combinatorics of $\A$ encodes the information of which singular point is on which line.
\end{rmk}

\begin{de}
	Let $\C$ be a combinatorics. A complex line arrangement $\A=\set{L_1,\cdots,L_n}$ of $\CC\PP^2$ is a \emph{realization} of $\C$  if its combinatorics agrees with $\C$. An \emph{ordered realization} of an ordered combinatorics is defined accordingly.
\end{de}

\begin{notation}
	If $\A$ is a realization of a combinatorics $\C$, then the incidence graph is also denoted by $\Gamma_\A$.
\end{notation}

\begin{ex}
 The incidence graph of a generic arrangement with 3 lines is the cyclic graph with 6 vertices. Its automorphism group is the dihedral group $\text{D}_3$.
\end{ex}

%%%%%%%%%   Défintion de l'arrangement qui nous intéresse
Using the fact that three lines are concurrent if and only if the determinant of their coefficients is null, it is simple to verify that:

\begin{propo}
	The arrangements defined by the following equations admit $\mathcal{K}$ as combinatorics:
	\begin{equation*}
		\begin{array}{l p{1cm} l }
			L_1:\ z=0 && L_2:\ x+y-z=0 \\
			L_3:\ x=0 && L_4:\ y=0 \\ 
			L_5:\ x-z=0 && L_6:\ y-z=0 \\
			L_7:\ -\alpha^3 x+z=0 && L_8:\ y-\alpha z=0 \\ 
			L_9:\ (\alpha-1) x-y+z=0 && L_{10}:\ -\alpha(\alpha-1) x+y+\alpha(\alpha-1) z=0 \\ 
			L_{11}:\ -\alpha(\alpha-1) x+y-\alpha z=0 &&
		\end{array}
	\end{equation*}
	where $\alpha$ is a root of the 10\up{th} cyclotomic polynomial $\Phi_{10}(X)=X^4-X^3+X^2-X+1$.
\end{propo}

We denote by $\mathcal{N}^+$, (resp. $\mathcal{M}^+$) the arrangement for which $\alpha\simeq -0.31 + 0.95 i$ (resp.  $\alpha\simeq 0.81 + 0.59 i
$), and  $\mathcal{N}^-$  (resp. $\mathcal{M}^-$) its complex conjugate arrangement.

\begin{rmk}
	The end of this paper (see Theorem~\ref{Thm_OrientedOrderedPair} and Subsection~\ref{Subsection_Computation}) will prove that they ($\mathcal{N}^+$, $\mathcal{M}^+$, $\mathcal{N}^-$ and $\mathcal{M}^-$) are representatives of the four connected components of the order moduli space, see~\cite{ArtCarCogMar_Topology}. Thus, these connected components admit representatives with complex equations in the ring of polynomials over the 10\up{th} cyclotomic fields. Their equations are linked by an element of the \textsc{Galois} group of $\QQ(\zeta_{10})$.
\end{rmk}

\begin{de}
	The \emph{topological type} of an arrangement $\A$ is the homeomorphism type of the pair $(\CC\PP^2,\A)$. If the homeomorphism preserves the orientation of $\CC\PP^2$, then we have \emph{oriented} topological type; and it is \emph{ordered}, if $\A$ is ordered and the homeomorphisms preserve this order.
\end{de}

\begin{rmk}
	If two ordered arrangements with the same combinatorics have different oriented and ordered topological type then they are in distinct ambient isotopy classes. The \textsc{MacLane} arrangements \cite{Mac} are the first such examples.
\end{rmk}

With these definitions, we can state the main results of the paper:

\begin{thm}\label{Thm_OrientedOrderedPair}
	There is no homeomorphism preserving both orientation and order between any two pairs among $(\CC\PP^2,\mathcal{N}^+)$, $(\CC\PP^2,\mathcal{N}^-)$, $(\CC\PP^2,\mathcal{M}^+)$ and $(\CC\PP^2,\mathcal{M}^-)$.
\end{thm}	

\begin{rmk}
	The complex conjugation induces an orientation-reversing homeomorphism between $(\CC\PP^2,\mathcal{N}^+)$ and $(\CC\PP^2,\mathcal{N}^-)$, and between $(\CC\PP^2,\mathcal{M}^+)$ and $(\CC\PP^2,\mathcal{M}^-)$ too.
\end{rmk}

\begin{cor}\label{Cor_OrderedPair}
	There is no order-preserving homeomorphism between $(\CC\PP^2,\mathcal{N}^+)$ (or $(\CC\PP^2,\mathcal{N}^-)$) and $(\CC\PP^2,\mathcal{M}^+)$ (or $(\CC\PP^2,\mathcal{M}^-)$).
\end{cor}

The proofs of both are presented in Subsection~\ref{Subsection_Computation}.

%%%%%%%%%%%%%%%%%%%%%%%%%%%%%%%%%%%%%%%%%%%%%%
%%%%%%%%%%%%%%%%%%%%%%%%%%%%%%%%%%%%%%%%%%%%%%
\section{Zariski pairs}\label{Section_Pairs}
%%%%%%%%%%%%%%%%%%%%%%%%%%%%%%%%%%%%%%%%%%%%%%
%%%%%%%%%%%%%%%%%%%%%%%%%%%%%%%%%%%%%%%%%%%%%%

The principal problem which appears while working with the previous combinatorics $\mathcal{K}$ is that its automorphism group is not trivial. Indeed, this group is cyclic of order 4 and is generated by the permutation:
\begin{equation*}
	\sigma=(1\ 3\ 2\ 4)(5\ 6)(7\ 9\ 10\ 8).
\end{equation*} 
This automorphism induces the automorphism by the \textsc{Galois} group of $\QQ(\zeta_{10})$: $a\mapsto -a^2$, where $a$ is a primitive root of unity. The change of variables $(x,y,z)\mapsto (z,x+y-z,y)$ realizes this automorphism of the combinatorics. It permutes cyclically the arrangements $\mathcal{N}^+$, $\mathcal{M}^+$, $\mathcal{N}^-$ and $\mathcal{M}^-$.

To remove the hypothesis ``order-preserving'' in Corollary~\ref{Cor_OrderedPair}, we use the same argument as in~\cite{ArtCarCogMar_Topology}: we consider additional lines to the previous combinatorics to reduce the automorphism group of the combinatorics to the trivial group. Let us consider the combinatorics $\mathfrak{K}=(\mathfrak{L},\mathfrak{P})$ obtained from $\mathcal{K}$ by adding a line $L_{12}$ at $\L$ passing through the point $L_1 \cap L_3 \cap L_5 \cap L_7$ and generic with the other lines, that is $\mathfrak{L}=\set{L_1,\cdots,L_{12}}$ and
\begin{multline*}
	\mathfrak{P} = \left\{\ \set{1, 2}, \set{1, 3, 5, 7, 12}, \set{1, 4, 6, 8}, \set{1, 9}, \set{1, 10, 11}, \set{2, 3, 6, 9},\right. \\
	\set{2, 4, 5, 10}, \set{2, 7, 11}, \set{2, 8}, \set{2, 12}, \set{3, 4}, \set{3, 8, 11}, \set{3, 10}, \set{4, 7}, \\
	\set{4, 9, 11}, \set{4, 12}, \set{5, 6}, \set{5, 8, 9}, \set{5, 11}, \set{6, 7, 10}, \set{6, 11}, \set{6, 12}, \\
	\left. \set{7, 8}, \set{7, 9}, \set{8, 10}, \set{8, 12}, \set{9, 10}, \set{9, 12}, \set{10, 12}, \set{11, 12}\ \right\}.
\end{multline*}
It admits four realizations denoted by $\mathfrak{N}^+$, $\mathfrak{N}^-$, $\mathfrak{M}^+$ and $\mathfrak{M}^-$ (in accordance with the realizations of $\mathcal{K}$).

\begin{propo}\label{Propo_AutomorphismGroup}
	The automorphism group of the combinatorics $\mathfrak{K}$ is trivial.
\end{propo}

\begin{proof}
	By construction, the line $L_{12}$ is the only line containing the point of multiplicity 5 and only double points. Then it is fixed by all automorphisms. Thus an automorphism of $\mathfrak{K}$ fixes the unique point of multiplicity 5. But in $\mathcal{K}$, the action of the automorphism group permutes cyclically the points of multiplicity 4. Since one of them was transformed into the unique point of multiplicity 5, then the automorphism group of $\mathfrak{K}$ is trivial.
\end{proof}

\begin{thm}\label{Theorem_ZariskiPair}
	There is no homeomorphism between $(\CC\PP^2,\mathfrak{N}^+)$ and $(\CC\PP^2,\mathfrak{M}^+)$.
\end{thm}

\begin{proof}
	By Corollary~\ref{Cor_OrderedPair}, there is no order-preserving homeomorphism between $(\CC\PP^2,\mathfrak{N}^+)$ and $(\CC\PP^2,\mathfrak{M}^+)$. Assume that there is a homeomorphism between $(\CC\PP^2,\mathfrak{N}^+)$ and $(\CC\PP^2,\mathfrak{M}^+)$ that does not preserve the order. Then it induces a non-trivial automorphism of the combinatorics $\mathfrak{K}$, which is impossible according to Proposition~\ref{Propo_AutomorphismGroup}.
\end{proof}

\begin{cor}\label{Cor_ZariskiPairs}
	There is no homeomorphism between $(\CC\PP^2,\mathfrak{N}^+)$ (or $(\CC\PP^2,\mathfrak{N}^-)$) and $(\CC\PP^2,\mathfrak{M}^+)$ (or  $(\CC\PP^2,\mathfrak{M}^-)$).
\end{cor}

We could remark that if the lines added to delete the automorphism of the combinatorics, are conjugated in~$\QQ(\zeta_{10})$ then the \textsc{Zariski} pairs obtained are arithmetic pairs. In particular, their fundamental group have the same profinite completion (\emph{i.e.} the same finite quotients). But if the lines are not conjugated in $\QQ(\zeta_{10})$ then the pairs obtained are not arithmetic. 

\begin{lem}\label{Lemma_OrientedPairs}
	There is no orientation-preserving homeomorphism between $(\CC\PP^2,\mathfrak{N}^+)$ and $(\CC\PP^2,\mathfrak{N}^-)$ (resp. $(\CC\PP^2,\mathfrak{M}^+)$ and $(\CC\PP^2,\mathfrak{M}^-)$).
\end{lem}

\begin{proof}
	By Theorem~\ref{Thm_OrientedOrderedPair}, there is no homeomorphism preserving both orientation and order between the pairs $(\CC\PP^2,\mathfrak{N}^+)$ and $(\CC\PP^2,\mathfrak{N}^-)$. But, by construction, there is no non-trivial automorphism of the combinatorics $\mathfrak{K}$. Then there is no orientation-preserving homeomorphism from $(\CC\PP^2,\mathfrak{N}^+)$ to $(\CC\PP^2,\mathfrak{N}^-)$.
\end{proof}

\begin{cor}\label{Cor_Zariski4tuplet}
	There is no orientation-preserving homeomorphism between any two pairs among  $(\CC\PP^2,\mathfrak{N}^+)$, $(\CC\PP^2,\mathfrak{N}^-)$, $(\CC\PP^2,\mathfrak{M}^+)$ and $(\CC\PP^2,\mathfrak{M}^-)$.
\end{cor}

\begin{proof}
	It is a consequence of Corollary~\ref{Cor_ZariskiPairs} and Lemma~\ref{Lemma_OrientedPairs}.
\end{proof}
%%%%%%%%%%%%%%%%%%%%%%%%%%%%%%%%%%%%%%%%%%%%%%
%%%%%%%%%%%%%%%%%%%%%%%%%%%%%%%%%%%%%%%%%%%%%%
\section{Oriented and ordered topological types}\label{Section_Distinction}
%%%%%%%%%%%%%%%%%%%%%%%%%%%%%%%%%%%%%%%%%%%%%%
%%%%%%%%%%%%%%%%%%%%%%%%%%%%%%%%%%%%%%%%%%%%%%

This section is the mathematical cornerstone of the paper. It contains (in Subsection~\ref{Subsection_Computation}) the distinction of the different ambient isotopy classes of the arrangements $\mathcal{N}^+$, $\mathcal{N}^-$, $\mathcal{M}^+$ and $\mathcal{M}^-$ previously constructed, and then the proof that they form an oriented and ordered \textsc{Zarsiki} 4-tuplet (i.e. Theorem~\ref{Thm_OrientedOrderedPair}).

%%%%%%%%%%%%%%%%%%%%%%%%%%%%%%%%%%%%%%%%%%%%%%
\subsection{The invariant $\mathcal{I}(\A,\xi,\gamma)$}\label{Subsection_Invariant}\mbox{}
%%%%%%%%%%%%%%%%%%%%%%%%%%%%%%%%%%%%%%%%%%%%%%

%%%%%%%%%%%%%%%%%%%%%%%%%%%%%%%%%%%%%%%%%%%%%%
\subsubsection{An inner-cyclic arrangement}\mbox{}
%%%%%%%%%%%%%%%%%%%%%%%%%%%%%%%%%%%%%%%%%%%%%%

For more details on the construction and for the computation of the invariant $\mathcal{I}(\A,\xi,\gamma)$ see~\cite{ArtFloGue}. We denote by $E_\A$ the complement of $\A=\set{L_1,\cdots,L_n}$ in $\CC\PP^2$. Let $\alpha_i\in \HH_1(E_\A)$ be the homological meridian associated with the line $L_i\in\A$. Remark that the set $\set{\alpha_2,\cdots,\alpha_n}$ generates $\HH_1(E_\A)$, indeed we have the relation $\alpha_1+\cdots+\alpha_n=0$. A \emph{character} on an arrangement $\A$ is a groups-morphism:
\begin{equation*}
	\xi:\HH_1(E_\A)\rightarrow \CC^*,
	\end{equation*}
	with $\prod\limits_{L_i\in\A} \xi(\alpha_i)=1$ to respect the previous relation.
	
%\begin{rmk}
	%Since a character $\xi$ is determined by its value on each line of $\A$, it can be defined on the combinatorics $\C_\A=(\L,\P)$ of $\A$ by assigning at each element of $\L$ a complex number. 
	%
	%depends only on the combinatorics of the arrangement. If $\alpha_i$ (resp. $\alpha'_i$) is the meridian around $L_i\in\A$ (resp. $L'_i\in\A'$), then the character $\xi$ is defined on both $\A$ and $\A'$ by $\xi(\alpha_i)=\xi(\alpha'_i)$.
%\end{rmk}

\begin{de}\label{Definition_InnerCyclic}
	Let $\A$ be an arrangement, $\xi$ be a character on $\A$ and $\gamma$ be a cycle of $\Gamma_\A$.  The triplet $(\A,\xi,\gamma)$ is an \emph{inner-cylic arrangement} if:
	\begin{enumerate}
		\item for all $v_{L_i}\in V_\L(\C)$, if $v_{L_i}\in \gamma$, then $\xi(\alpha_i)=1$,
		\item for all $v_P\in V_\P(\C)$, if $v_P\in \gamma$, then for all $L_i\ni P$, $\xi(\alpha_i)=1$,
		\item for all $v_L\in\gamma$, if $P\in L$ then $\prod\limits_{L_i\ni P}\xi(\alpha_i)=1$.
	\end{enumerate}
\end{de}

\begin{rmk}
	Suppose that $\A$ and $\A'$ are two realizations of the same combinatorics (\emph{i.e.} there is an isomorphism $\phi : \C_\A \xrightarrow{\sim} \C_{\A'}$ of ordered combinatorics). If $\xi$ is a character on $E_\A$, then it induces on $E_{\A'}$ a character $\xi'$ defined by $\xi'\circ \phi = \xi$. Furthermore, if $(\A,\xi,\gamma)$ is an inner-cyclic arrangement, then $(\A',\xi',\gamma')$ is an inner-cyclic arrangement too, where $\gamma'$ is the cycle of $\Gamma_{\A'}$ obtained from $\gamma$ by $\phi$. In other terms the existence of a character $\xi$ and a cycle $\gamma$ such that $(\A,\xi,\gamma)$ is an inner-cyclic arrangement, is determined by the combinatorics of $\A$.
\end{rmk}

By the previous remark, we can define a character on the combinatorics $\mathcal{K}$ and consider it on $\mathcal{N}^+$, $\mathcal{N}^-$, $\mathcal{M}^+$ and $\mathcal{M}^-$. Thus let $\xi$ be such a character defined as follows:
\begin{equation*}
		\xi:(L_1,\cdots,L_{11}) \longmapsto (\zeta,\zeta^4,\zeta^3,\zeta^2,1,1,\zeta,\zeta^2,\zeta^3,\zeta^4,1),
\end{equation*}
where $\zeta$ is a primitive 5-root of unity. Let $\gamma_{(5,6,11)}$ be the cycle of $\HH_1(\Gamma_\mathcal{K})$ defined by:
\begin{equation*}
	\begin{tikzpicture}[xscale=1.5]
		\node (A1) at (0,0) {$v_{L_5}$};
		\node (A2) at (1,0) {$v_{P_{5,6}}$};
		\node (A3) at (2,0) {$v_{L_6}$};
		\node (A4) at (3,0) {$v_{P_{6,11}}$};
		\node (A5) at (4,0) {$v_{L_{11}}$};
		\node (A6) at (5,0) {$v_{P_{11,5}}$};
		\draw[->] (A1)  -- (A2);
		\draw[->] (A2) -- (A3);
		\draw[->] (A3) -- (A4);
		\draw[->] (A4) -- (A5);
		\draw[->] (A5) -- (A6);
		\draw[->] (A6) to[out=110,in=0] (3.5,0.75) -- (1.5,0.75) to[out=180,in=80] (A1);
	\end{tikzpicture}
\end{equation*}

\begin{propo}\label{Propo_N_M_inner_cylic}
	The triplets $(\mathcal{N}^+,\xi,\gamma_{(5,6,11)})$, $(\mathcal{N}^-,\xi,\gamma_{(5,6,11)})$, $(\mathcal{M}^+,\xi,\gamma_{(5,6,11)})$ and $(\mathcal{M}^-,\xi,\gamma_{(5,6,11)})$ are inner-cyclic arrangements.
\end{propo}

The proof of this proposition is straightforward, since the three conditions of Definition~\ref{Definition_InnerCyclic} are combinatorial. 

%%%%%%%%%%%%%%%%%%%%%%%%%%%%%%%%%%%%%%%%%%%%%%
\subsubsection{Construction of the invariant}\mbox{}
%%%%%%%%%%%%%%%%%%%%%%%%%%%%%%%%%%%%%%%%%%%%%%

The \emph{boundary manifold} $B_\A$ is the common boundary of a regular neighbourhood of the arrangement $\A$: $\Tub(\A)=\big(\bigcup \B(P)\big) \bigcup \big(\bigcup\Tub(L) \big) $ (where $\B(P)$ are 4-balls centered in the singular points of $\A$) and the exterior $\CC\PP^2\setminus \overset{\circ}{\Tub}(\A)$.  Up to homotopy type, there is a natural projection $\Tub(\A) \overset{\rho}{\longrightarrow} \Gamma_\A$, which induces an isomorphism $\rho_*$ on the first homology groups. A \emph{holed neighbourhood} $\Tub(\gamma)$ associated with $\gamma$ is a sub-manifold of $\Tub(\A)$ of the form:
\begin{equation*}
	\Tub(\gamma)=\Tub(\A)\setminus\Big[ \big( \bigcup\limits_{v_L\notin\gamma}\Tub( L) \big) \cup \big( \bigcup\limits_{v_P\notin\gamma}\B(P) \big) \Big].
\end{equation*}
A \emph{nearby cycle} $\tilde{\gamma}\in\HH_1(B_\A)$ associated with a cycle $\gamma\in\HH_1(\Gamma_\A)$ is defined as a path of $\partial\big(\overline{\Tub(\gamma)}\big)$ isotopic to $\rho_*\inv(\gamma)$ in $\Tub(\A)$.

We denote by $i$ the inclusion map of the boundary manifold in the complement $i : B_\A \hookrightarrow E_\A$.  Let $\A$ be a realization of $\C$, and $\xi$ be a torsion character on $\A$. We consider the following map:
\begin{equation*}
	\chi_{(\A,\xi)}:
		 \HH_1(B_\A)  \overset{i_*}{\longrightarrow}  \HH_1(E_\A)  \overset{\xi}{\longrightarrow}  \CC^*.
\end{equation*}
If $(\A,\xi,\gamma)$ is inner-cyclic and $\tilde{\gamma}$ is a nearby cycle associated with $\gamma$, then the value of $\chi_{(\A,\xi)} (\widetilde{\gamma})$ is independent of the choice of the nearby cycle $\tilde{\gamma}$, see~\cite[Lemma 2.2]{ArtFloGue}. Thus we define:
\begin{equation*}
 \mathcal{I}(\A,\xi, \gamma) = \chi_{(\A,\xi)} (\widetilde{\gamma}),
\end{equation*}
where $\tilde{\gamma}\in\HH_1(B_\A)$ is any nearby cycle associated with $\gamma$.

\begin{thm}[\cite{ArtFloGue}]\label{Thm_TopologicalInvariant}
	Let $\A$ and $\A'$ be two ordered realizations of an ordered combinatorics $\C$. If $(\A,\xi,\gamma)$  and $(\A',\xi,\gamma)$ are two inner-cyclic arrangements with the same oriented and ordered topological type, then 
	\begin{equation*}
		\mathcal{I}(\A,\xi, \gamma) = \mathcal{I}(\A',\xi,\gamma).
	\end{equation*}
\end{thm}

%%%%%%%%%%%%%%%%%%%%%%%%%%%%%%%%%%%%%%%%%%%%%%
\subsection{Computation of the invariant}\label{Subsection_Computation}\mbox{}
%%%%%%%%%%%%%%%%%%%%%%%%%%%%%%%%%%%%%%%%%%%%%%

%%%%%%%%%%%%%%%%%%%%%%%%%%%%%%%%%%%%%%%%%%%%%%
\subsubsection{Braided wiring diagrams}\mbox{}
%%%%%%%%%%%%%%%%%%%%%%%%%%%%%%%%%%%%%%%%%%%%%%

The invariant $\mathcal{I}(\A,\xi,\gamma)$ can be computed from the \emph{braided wiring diagram} of the arrangement. It is a singular braid associated with the arrangement (for more details see~\cite{Arv,CohSuc_Monodromy}) and it is defined as follows: consider a line $L\in\A$ as the line at the infinity, and let $\A\aff\subset\CC^2\simeq\CC\PP^2\setminus\set{L}$ be the associated affine arrangement. Let $p:\CC^2\rightarrow\CC$ be a generic projection for the arrangement $\A\aff$ (i.e. no line of $\A\aff$ is a fiber of~$p$). If $\nu:[0,1]\rightarrow\CC$ is a smooth path containing the images of the singular points of $\A\aff$ by $p$ (or continuous and piecewise smooth if the images of the singularities are in the smooth part of $\nu$), then we define:

\begin{de}
	The \emph{braided wiring diagram} of $\A$ associated with $\nu$ and $p$ is defined by:
	\begin{equation*}
		W_\A=\set{(x,y)\in\A\aff\ \mid\ p(x,y)\in\nu([0,1])}.
	\end{equation*}
	The trace $\omega_i= W_\A \cap L_i$ is called the \emph{wire} associated to the line $L_i$.
\end{de}

	\begin{figure}[h!]
		\centering
			\subfloat[The path $\nu_{\mathcal{N}^+}$]{\label{nu_N}\includegraphics[width=7.5cm,height=6.5cm]{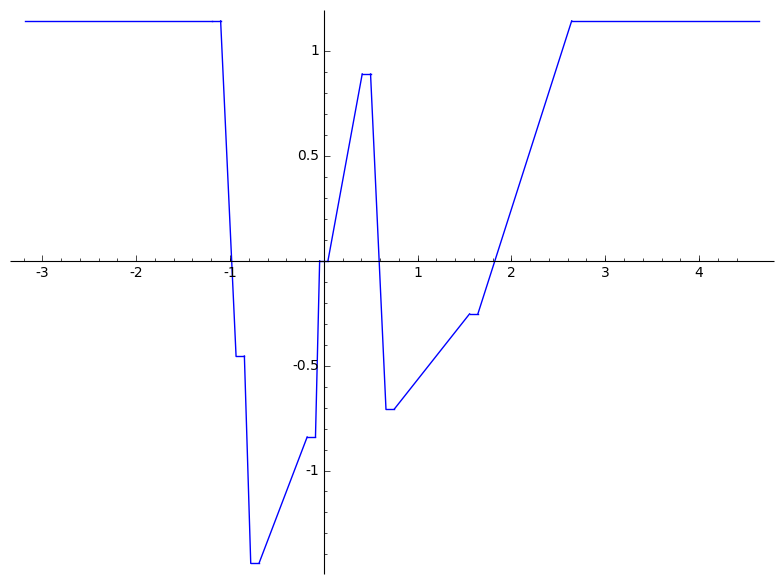}}
			\hspace{0.5cm}
			\subfloat[The path $\nu_{\mathcal{M}^+}$]{\label{nu_M}\includegraphics[width=7.5cm,height=6.5cm]{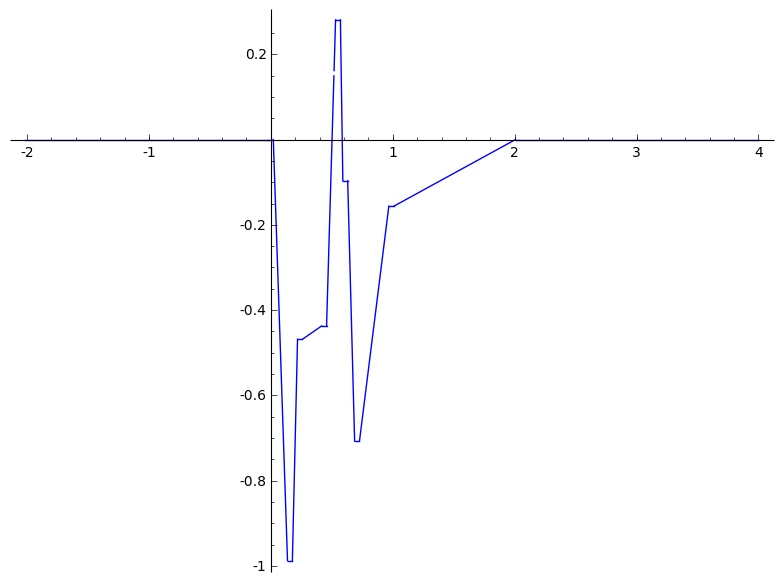}}			
		\caption{ Paths $\nu$ use for the computation of the wiring diagrams \label{nu_paths}}
	\end{figure}

From the equations of  $\mathcal{N}^+$ and $\mathcal{M}^+$, we compute their wiring diagrams. To use the result on the computation of $i_*:\p(B_\A)\rightarrow \p(E_\A)$ developed in~\cite{FloGueMar}, we choose a line supporting the cycle $\gamma$ as the line at infinity. With the change of variables $x\mapsto\lambda x$ and $z\mapsto\lambda x - z$, where $\lambda=e^{i\pi/4}$,  the line $L_5$ is consider as the line at infinity for the projection $p:[x:y:z]\rightarrow (x/z,y/z)$. Note that with this change of variables the lines $L_1$, $L_3$ and $L_7$ are vertical (\emph{i.e.} fibers of the projection), so the projection is not generic. Nevertheless, we can draw the wiring diagram of $\A\setminus\set{L_1,L_3,L_7}$, and adding these lines as vertical one, see~Figures~\ref{11_prime_A},~\ref{11_prime_B}, and obtain non-generic braided wiring diagram. The paths $\nu_{\mathcal{N}^+}$ and $\nu_{\mathcal{M}^+}$ considered to obtain these diagrams are pictured in  Figure~\ref{nu_paths}. All these computations are done using \texttt{Sage}~\cite{Sage}. The source can be downloaded from the author's website:\nopagebreak
\begin{center}\nopagebreak\texttt{http://www.benoit-guervilleballe.com/publications/ZP\_wiring\_diagrams.zip}\end{center}

\begin{rmk}
	For smaller and clearer pictures, some braids are simplified (using Reidemeister moves) in the non-generic wiring diagrams of $\mathcal{N}^+$ and $\mathcal{M}^+$ pictured in Figure~\ref{11_prime_A} and Figure~\ref{11_prime_B} .
\end{rmk}

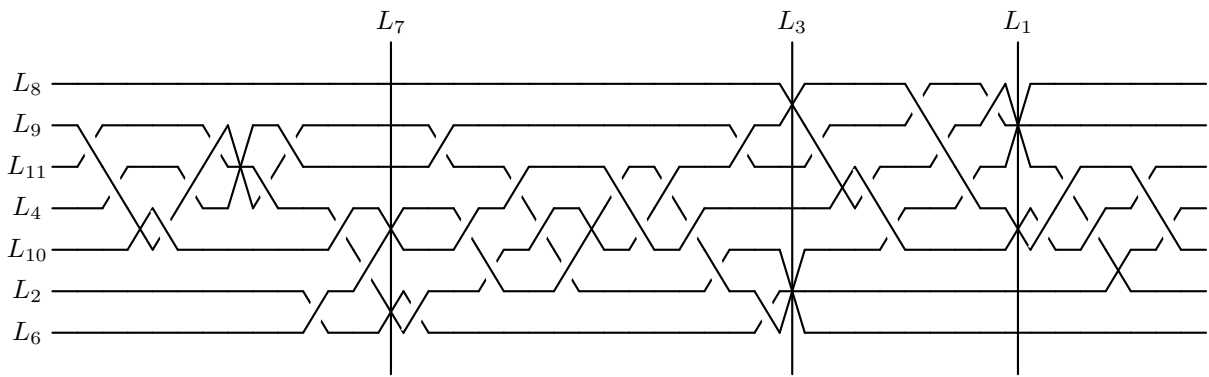
\begin{figure}[h!]
	\centering
	\begin{tikzpicture}
	\begin{scope}[xscale=0.33,yscale=0.55]
		\ncross{0}{0};\ncross{1}{0};\ncross{2}{0};\ncross{3}{0};\ncross{4}{0};\ncross{5}{0};\ncross{6}{0};
		\ncross{0}{1};\ncross{1}{1};\ncross{2}{1};\ncross{3}{1};\ocross{4}{1};\ncross{6}{1};
		\ncross{0}{2};\ncross{1}{2};\ncross{2}{2};\ocross{3}{2};\ncross{5}{2};\ncross{6}{2};
		\ncross{0}{3};\ncross{1}{3};\rcross{2}{3};\ncross{4}{3};\ncross{5}{3};\ncross{6}{3};
		\ncross{0}{4};\ncross{1}{4};\ocross{2}{4};\ncross{4}{4};\ncross{5}{4};\ncross{6}{4};
		\ncross{0}{5};\ncross{1}{5};\ncross{2}{5};\ucross{3}{5};\ncross{5}{5};\ncross{6}{5};
		\ncross{0}{6};\ncross{1}{6};\ncross{2}{6};\ncross{3}{6};\ucross{4}{6};\ncross{6}{6};
		\ncross{0}{7};\ncross{1}{7};\ncross{2}{7};\mcross{3}{3}{7};\ncross{6}{7};
		\ncross{0}{8};\ncross{1}{8};\ncross{2}{8};\ocross{3}{8};\ncross{5}{8};\ncross{6}{8};
		\ncross{0}{9};\ncross{1}{9};\ncross{2}{9};\ncross{3}{9};\ocross{4}{9};\ncross{6}{9};
		\ucross{0}{10};\ncross{2}{10};\ncross{3}{10};\ncross{4}{10};\ncross{5}{10};\ncross{6}{10};
		\ncross{0}{11};\ncross{1}{11};\ucross{2}{11};\ncross{4}{11};\ncross{5}{11};\ncross{6}{11};
		\ncross{0}{12};\ucross{1}{12};\ncross{3}{12};\ncross{4}{12};\ncross{5}{12};\ncross{6}{12};
		\rcross{0}{13};\rcross{2}{13};\ncross{4}{13};\ncross{5}{13};\ncross{6}{13};
		\draw[thick,cap=round] (13.5,-1) -- (13.5,7);
		\ucross{0}{14};\ncross{2}{14};\ncross{3}{14};\ncross{4}{14};\ncross{5}{14};\ncross{6}{14};
		\ncross{0}{15};\ncross{1}{15};\ncross{2}{15};\ncross{3}{15};\ucross{4}{15};\ncross{6}{15};
		\ncross{0}{16};\ncross{1}{16};\ucross{2}{16};\ncross{4}{16};\ncross{5}{16};\ncross{6}{16};
		\ncross{0}{17};\ocross{1}{17};\ncross{3}{17};\ncross{4}{17};\ncross{5}{17};\ncross{6}{17};
		\ncross{0}{18};\ncross{1}{18};\ncross{2}{18};\ucross{3}{18};\ncross{5}{18};\ncross{6}{18};
		\ncross{0}{19};\ncross{1}{19};\ucross{2}{19};\ncross{4}{19};\ncross{5}{19};\ncross{6}{19};
		\ncross{0}{20};\ucross{1}{20};\ncross{3}{20};\ncross{4}{20};\ncross{5}{20};\ncross{6}{20};
		\ncross{0}{21};\ncross{1}{21};\rcross{2}{21};\ncross{4}{21};\ncross{5}{21};\ncross{6}{21};
		\ncross{0}{22};\ncross{1}{22};\ncross{2}{22};\ocross{3}{22};\ncross{5}{22};\ncross{6}{22};
		\ncross{0}{23};\ncross{1}{23};\ocross{2}{23};\ncross{4}{23};\ncross{5}{23};\ncross{6}{23};
		\ncross{0}{24};\ncross{1}{24};\ncross{2}{24};\ucross{3}{24};\ncross{5}{24};\ncross{6}{24};
		\ncross{0}{25};\ncross{1}{25};\ucross{2}{25};\ncross{4}{25};\ncross{5}{25};\ncross{6}{25};
		\ncross{0}{26};\ocross{1}{26};\ncross{3}{26};\ncross{4}{26};\ncross{5}{26};\ncross{6}{26};
		\ncross{0}{27};\ncross{1}{27};\ncross{2}{27};\ncross{3}{27};\ucross{4}{27};\ncross{6}{27};
		\ocross{0}{28};\ncross{2}{28};\ncross{3}{28};\ncross{4}{28};\ncross{5}{28};\ncross{6}{28};
		\mcross{0}{3}{29};\ncross{3}{29};\ncross{4}{29};\rcross{5}{29};
		\draw[thick,cap=round] (29.5,-1) -- (29.5,7);
		\ncross{0}{30};\ncross{1}{30};\ncross{2}{30};\ncross{3}{30};\ocross{4}{30};\ncross{6}{30};
		\ncross{0}{31};\ncross{1}{31};\ncross{2}{31};\rcross{3}{31};\ncross{5}{31};\ncross{6}{31};
		\ncross{0}{32};\ncross{1}{32};\ncross{2}{32};\ocross{3}{32};\ncross{5}{32};\ncross{6}{32};
		\ncross{0}{33};\ncross{1}{33};\ocross{2}{33};\ncross{4}{33};\ncross{5}{33};\ncross{6}{33};
		\ncross{0}{34};\ncross{1}{34};\ncross{2}{34};\ncross{3}{34};\ncross{4}{34};\ocross{5}{34};
		\ncross{0}{35};\ncross{1}{35};\ncross{2}{35};\ncross{3}{35};\ocross{4}{35};\ncross{6}{35};
		\ncross{0}{36};\ncross{1}{36};\ncross{2}{36};\ocross{3}{36};\ncross{5}{36};\ncross{6}{36};
		\ncross{0}{37};\ncross{1}{37};\ncross{2}{37};\ncross{3}{37};\ncross{4}{37};\ucross{5}{37};
		\ncross{0}{38};\ncross{1}{38};\rcross{2}{38};\mcross{4}{3}{38};
		\draw[thick,cap=round] (38.5,-1) -- (38.5,7);		
		\ncross{0}{39};\ncross{1}{39};\ucross{2}{39};\ncross{4}{39};\ncross{5}{39};\ncross{6}{39};
		\ncross{0}{40};\ncross{1}{40};\ncross{2}{40};\ucross{3}{40};\ncross{5}{40};\ncross{6}{40};
		\ncross{0}{41};\ncross{1}{41};\ucross{2}{41};\ncross{4}{41};\ncross{5}{41};\ncross{6}{41};
		\ncross{0}{42};\rcross{1}{42};\ncross{3}{42};\ncross{4}{42};\ncross{5}{42};\ncross{6}{42};
		\ncross{0}{43};\ncross{1}{43};\ncross{2}{43};\ocross{3}{43};\ncross{5}{43};\ncross{6}{43};
		\ncross{0}{44};\ncross{1}{44};\ocross{2}{44};\ncross{4}{44};\ncross{5}{44};\ncross{6}{44};
		\ncross{0}{45};\ncross{1}{45};\ncross{2}{45};\ncross{3}{45};\ncross{4}{45};\ncross{5}{45};\ncross{6}{45};

		\node at (-1,6) {$L_{8}$};
		\node at (-1,5) {$L_{9}$};
		\node at (-1,4) {$L_{11}$};
		\node at (-1,3) {$L_{4}$};
		\node at (-1,2) {$L_{10}$};
		\node at (-1,1) {$L_{2}$};
		\node at (-1,0) {$L_{6}$};
		
		\node at (13.5,7.5) {$L_{7}$};
		\node at (29.5,7.5) {$L_{3}$};
		\node at (38.5,7.5) {$L_{1}$};
	\end{scope}
\end{tikzpicture}
	\caption{Non-generic braided wiring diagram of $\mathcal{N}^+$ \label{11_prime_A}}
\end{figure}

\begin{figure}[h!]
	\centering
	\begin{tikzpicture}
	\begin{scope}[xscale=0.32,yscale=0.55]
  	
		\ncross{0}{0};\ncross{1}{0};\ncross{2}{0};\ncross{3}{0};\ncross{4}{0};\ncross{5}{0};\ncross{6}{0};
		\ncross{0}{1};\ncross{1}{1};\ncross{2}{1};\ncross{3}{1};\ncross{4}{1};\ocross{5}{1};
		\ncross{0}{2};\ncross{1}{2};\ucross{2}{2};\ncross{4}{2};\ncross{5}{2};\ncross{6}{2};
		\ncross{0}{3};\ncross{1}{3};\ncross{2}{3};\ncross{3}{3};\ocross{4}{3};\ncross{6}{3};
		\mcross{0}{3}{4};\rcross{3}{4};\ncross{5}{4};7\ncross{6}{4};
		\draw[thick,cap=round] (4.5,-1) -- (4.5,7);
		\ncross{0}{5};\ocross{1}{5};\ncross{3}{5};\ncross{4}{5};\ncross{5}{5};\ncross{6}{5};
		\ncross{0}{6};\ncross{1}{6};\ncross{2}{6};\ncross{3}{6};\ucross{4}{6};\ncross{6}{6};
		\ncross{0}{7};\ncross{1}{7};\ncross{2}{7};\ucross{3}{7};\ncross{5}{7};\ncross{6}{7};
		\ncross{0}{8};\ncross{1}{8};\ucross{2}{8};\ncross{4}{8};\ncross{5}{8};\ncross{6}{8};
		\ncross{0}{9};\ncross{1}{9};\ncross{2}{9};\ncross{3}{9};\ncross{4}{9};\ucross{5}{9};
		\ncross{0}{10};\ncross{1}{10};\ncross{2}{10};\ncross{3}{10};\ucross{4}{10};\ncross{6}{10};
		\ncross{0}{11};\ncross{1}{11};\ncross{2}{11};\mcross{3}{3}{11};\ncross{6}{11};
		\ncross{0}{12};\ncross{1}{12};\ncross{2}{12};\ncross{3}{12};\ucross{4}{12};\ncross{6}{12};
		\ncross{0}{13};\ncross{1}{13};\ncross{2}{13};\ncross{3}{13};\ncross{4}{13};\ocross{5}{13};
		\ncross{0}{14};\ncross{1}{14};\ncross{2}{14};\ucross{3}{14};\ncross{5}{14};\ncross{6}{14};
		\ncross{0}{15};\ncross{1}{15};\ocross{2}{15};\ncross{4}{15};\ncross{5}{15};\ncross{6}{15};
		\ncross{0}{16};\ncross{1}{16};\ncross{2}{16};\ocross{3}{16};\ncross{5}{16};\ncross{6}{16};
		\ncross{0}{17};\ncross{1}{17};\ocross{2}{17};\ncross{4}{17};\ncross{5}{17};\ncross{6}{17};
		\ncross{0}{18};\ncross{1}{18};\ncross{2}{18};\ucross{3}{18};\ncross{5}{18};\ncross{6}{18};
		\ncross{0}{19};\rcross{1}{19};\ncross{3}{19};\ncross{4}{19};\ncross{5}{19};\ncross{6}{19};
		\ncross{0}{20};\ncross{1}{20};\ucross{2}{20};\ncross{4}{20};\ncross{5}{20};\ncross{6}{20};
		\ncross{0}{21};\ncross{1}{21};\ncross{2}{21};\ucross{3}{21};\ncross{5}{21};\ncross{6}{21};
		\ncross{0}{22};\ncross{1}{22};\ocross{2}{22};\ncross{4}{22};\ncross{5}{22};\ncross{6}{22};
		\ncross{0}{23};\ncross{1}{23};\rcross{2}{23};\ncross{4}{23};\ncross{5}{23};\ncross{6}{23};
		\ncross{0}{24};\ncross{1}{24};\ucross{2}{24};\ncross{4}{24};\ncross{5}{24};\ncross{6}{24};
		\ncross{0}{25};\ncross{1}{25};\ncross{2}{25};\ncross{3}{25};\ocross{4}{25};\ncross{6}{25};
		\ncross{0}{26};\ncross{1}{26};\ncross{2}{26};\ocross{3}{26};\ncross{5}{26};\ncross{6}{26};
		\ncross{0}{27};\ncross{1}{27};\ncross{2}{27};\ncross{3}{27};\ocross{4}{27};\ncross{6}{27};
		\ocross{0}{28};\ncross{2}{28};\ncross{3}{28};\ncross{4}{28};\ncross{5}{28};\ncross{6}{28};
		\ncross{0}{29};\ucross{1}{29};\ncross{3}{29};\ncross{4}{29};\ncross{5}{29};\ncross{6}{29};
		\ncross{0}{30};\ncross{1}{30};\ocross{2}{30};\ncross{4}{30};\ncross{5}{30};\ncross{6}{30};
		\ncross{0}{31};\ncross{1}{31};\rcross{2}{31};\ncross{4}{31};\ncross{5}{31};\ncross{6}{31};
		\ncross{0}{32};\ocross{1}{32};\ncross{3}{32};\ncross{4}{32};\ncross{5}{32};\ncross{6}{32};
		\ucross{0}{33};\ncross{2}{33};\ncross{3}{33};\ncross{4}{33};\ncross{5}{33};\ncross{6}{33};		
		\rcross{0}{34};\ncross{2}{34};\ncross{3}{34};\rcross{4}{34};\ncross{6}{34};
		\draw[thick,cap=round] (34.5,-1) -- (34.5,7);
		\ocross{0}{35};\ncross{2}{35};\ncross{3}{35};\ncross{4}{35};\ncross{5}{35};\ncross{6}{35};
		\ncross{0}{36};\ncross{1}{36};\ncross{2}{36};\ucross{3}{36};\ncross{5}{36};\ncross{6}{36};
		\ncross{0}{37};\ncross{1}{37};\ocross{2}{37};\ncross{4}{37};\ncross{5}{37};\ncross{6}{37};
		\ncross{0}{38};\ncross{1}{38};\ncross{2}{38};\ncross{3}{38};\ucross{4}{38};\ncross{6}{38};
		\ncross{0}{39};\ucross{1}{39};\ncross{3}{39};\ncross{4}{39};\ncross{5}{39};\ncross{6}{39};
		\ncross{0}{40};\rcross{1}{40};\ncross{3}{40};\mcross{4}{3}{40};
		\draw[thick,cap=round] (40.5,-1) -- (40.5,7);
		\ncross{0}{41};\ncross{1}{41};\ncross{2}{41};\ncross{3}{41};\ucross{4}{41};\ncross{6}{41};
		\ncross{0}{42};\ncross{1}{42};\ucross{2}{42};\ncross{4}{42};\ncross{5}{42};\ncross{6}{42};
		\ucross{0}{43};\ncross{2}{43};\ncross{3}{43};\ncross{4}{43};\ncross{5}{43};\ncross{6}{43};
		\ncross{0}{44};\ncross{1}{44};\ncross{2}{44};\ocross{3}{44};\ncross{5}{44};\ncross{6}{44};
		\ncross{0}{45};\ncross{1}{45};\ncross{2}{45};\rcross{3}{45};\ncross{5}{45};\ncross{6}{45};
		\ncross{0}{46};\ncross{1}{46};\ocross{2}{46};\ncross{4}{46};\ncross{5}{46};\ncross{6}{46};
		\ncross{0}{47};\ncross{1}{47};\ncross{2}{47};\ncross{3}{47};\ocross{4}{47};\ncross{6}{47};
		\ncross{0}{48};\ncross{1}{48};\ncross{2}{48};\ocross{3}{48};\ncross{5}{48};\ncross{6}{48};
		\ncross{0}{49};\ncross{1}{49};\ncross{2}{49};\ncross{3}{49};\ncross{4}{49};\ncross{5}{49};\ncross{6}{49};
		
		\node at (-1,6) {$L_{11}$};
		\node at (-1,5) {$L_{4}$};
		\node at (-1,4) {$L_{10}$};
		\node at (-1,3) {$L_{2}$};
		\node at (-1,2) {$L_{8}$};
		\node at (-1,1) {$L_{9}$};
		\node at (-1,0) {$L_{6}$};
		
		\node at (4.5,7.5) {$L_{3}$};
		\node at (34.5,7.5) {$L_{7}$};
		\node at (40.5,7.5) {$L_{1}$};

	\end{scope}
\end{tikzpicture}
	\caption{Non-generic braided wiring diagram of $\mathcal{M}^+$ \label{11_prime_B}}
\end{figure}
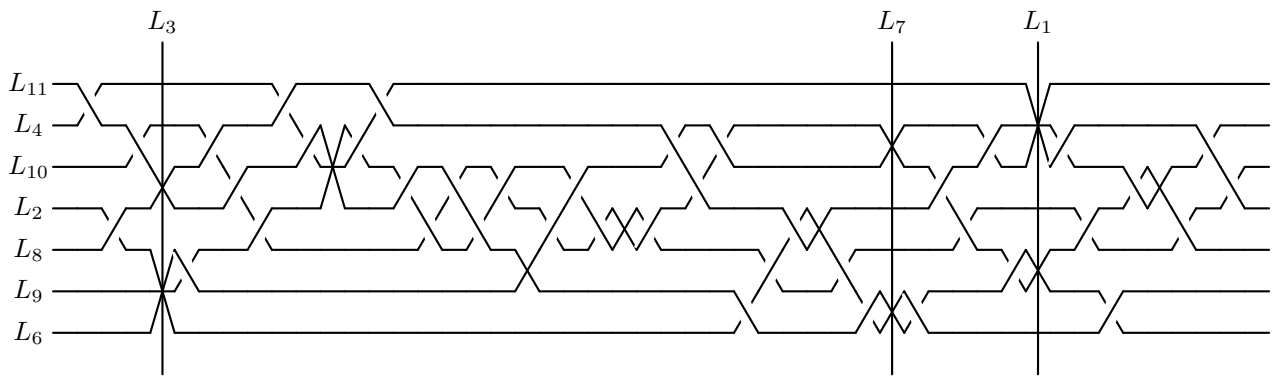

To obtain the wiring diagrams, we slightly modify the center of the projection $p$, such that it is always on $L_5$ and distinct (but very close) to the intersection of the lines $L_1$, $L_3$, $L_5$ and $L_7$. For example, a generic braided wiring diagram of $\mathcal{N}^+$ is pictured in Figure~\ref{Generic_WD_N+}, the perturbation applied to obtain this wiring diagram is such that the three vertical lines have a very negative slope and are still parallel (since their intersection point is still on the line at infinity). 

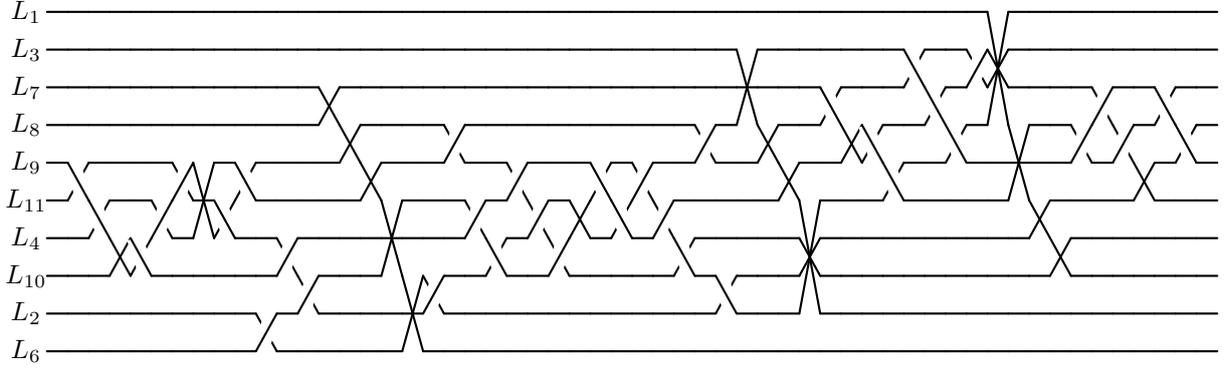
\begin{figure}[h!]
	\centering
	\begin{tikzpicture}
	\begin{scope}[xscale=0.275,yscale=0.5]
		\begin{scope}
			\ncross{0}{0};\ncross{1}{0};\ncross{2}{0};\ncross{3}{0};\ncross{4}{0};\ncross{5}{0};\ncross{6}{0};\ncross{7}{0};\ncross{8}{0};\ncross{9}{0};
			\ncross{0}{1};\ncross{1}{1};\ncross{2}{1};\ncross{3}{1};\ocross{4}{1};\ncross{6}{1};\ncross{7}{1};\ncross{8}{1};\ncross{9}{1};
			\ncross{0}{2};\ncross{1}{2};\ncross{2}{2};\ocross{3}{2};\ncross{5}{2};\ncross{6}{2};\ncross{7}{2};\ncross{8}{2};\ncross{9}{2};
			\ncross{0}{3};\ncross{1}{3};\rcross{2}{3};\ncross{4}{3};\ncross{5}{3};\ncross{6}{3};\ncross{7}{3};\ncross{8}{3};\ncross{9}{3};
			\ncross{0}{4};\ncross{1}{4};\ocross{2}{4};\ncross{4}{4};\ncross{5}{4};\ncross{6}{4};\ncross{7}{4};\ncross{8}{4};\ncross{9}{4};
			\ncross{0}{5};\ncross{1}{5};\ncross{2}{5};\ucross{3}{5};\ncross{5}{5};\ncross{6}{5};\ncross{7}{5};\ncross{8}{5};\ncross{9}{5};
			\ncross{0}{6};\ncross{1}{6};\ncross{2}{6};\ncross{3}{6};\ucross{4}{6};\ncross{6}{6};\ncross{7}{6};\ncross{8}{6};\ncross{9}{6};
			\ncross{0}{7};\ncross{1}{7};\ncross{2}{7};\mcross{3}{3}{7};\ncross{6}{7};\ncross{7}{7};\ncross{8}{7};\ncross{9}{7};
			\ncross{0}{8};\ncross{1}{8};\ncross{2}{8};\ocross{3}{8};\ncross{5}{8};\ncross{6}{8};\ncross{7}{8};\ncross{8}{8};\ncross{9}{8};
			\ncross{0}{9};\ncross{1}{9};\ncross{2}{9};\ncross{3}{9};\ocross{4}{9};\ncross{6}{9};\ncross{7}{9};\ncross{8}{9};\ncross{9}{9};
			\ucross{0}{10};\ncross{2}{10};\ncross{3}{10};\ncross{4}{10};\ncross{5}{10};\ncross{6}{10};\ncross{7}{10};\ncross{8}{10};\ncross{9}{10};
			\ncross{0}{11};\ncross{1}{11};\ucross{2}{11};\ncross{4}{11};\ncross{5}{11};\ncross{6}{11};\ncross{7}{11};\ncross{8}{11};\ncross{9}{11};
			\ncross{0}{12};\ucross{1}{12};\ncross{3}{12};\ncross{4}{12};\ncross{5}{12};\ncross{6}{12};\ncross{7}{12};\ncross{8}{12};\ncross{9}{12};

			\ncross{0}{13};\ncross{1}{13};\ncross{2}{13};\ncross{3}{13};\ncross{4}{13};\ncross{5}{13};\rcross{6}{13};\ncross{8}{13};\ncross{9}{13};
			\ncross{0}{14};\ncross{1}{14};\ncross{2}{14};\ncross{3}{14};\ncross{4}{14};\rcross{5}{14};\ncross{7}{14};\ncross{8}{14};\ncross{9}{14};
			\ncross{0}{15};\ncross{1}{15};\ncross{2}{15};\ncross{3}{15};\rcross{4}{15};\ncross{6}{15};\ncross{7}{15};\ncross{8}{15};\ncross{9}{15};
			\ncross{0}{16};\ncross{1}{16};\mcross{2}{3}{16};\ncross{5}{16};\ncross{6}{16};\ncross{7}{16};\ncross{8}{16};\ncross{9}{16};
			\mcross{0}{3}{17};\ncross{3}{17};\ncross{4}{17};\ncross{5}{17};\ncross{6}{17};\ncross{7}{17};\ncross{8}{17};\ncross{9}{17};
		\end{scope}
		
		\begin{scope}[shift={(4,1)}]
			\ucross{0}{14};\ncross{2}{14};\ncross{3}{14};\ncross{4}{14};\ncross{5}{14};\ncross{6}{14};\ncross{7}{14};\ncross{8}{14};\ncross{-1}{14};
			\ncross{0}{15};\ncross{1}{15};\ncross{2}{15};\ncross{3}{15};\ucross{4}{15};\ncross{6}{15};\ncross{7}{15};\ncross{8}{15};\ncross{-1}{15};
			\ncross{0}{16};\ncross{1}{16};\ucross{2}{16};\ncross{4}{16};\ncross{5}{16};\ncross{6}{16};\ncross{7}{16};\ncross{8}{16};\ncross{-1}{16};
			\ncross{0}{17};\ocross{1}{17};\ncross{3}{17};\ncross{4}{17};\ncross{5}{17};\ncross{6}{17};\ncross{7}{17};\ncross{8}{17};\ncross{-1}{17};
			\ncross{0}{18};\ncross{1}{18};\ncross{2}{18};\ucross{3}{18};\ncross{5}{18};\ncross{6}{18};\ncross{7}{18};\ncross{8}{18};\ncross{-1}{18};
			\ncross{0}{19};\ncross{1}{19};\ucross{2}{19};\ncross{4}{19};\ncross{5}{19};\ncross{6}{19};\ncross{7}{19};\ncross{8}{19};\ncross{-1}{19};
			\ncross{0}{20};\ucross{1}{20};\ncross{3}{20};\ncross{4}{20};\ncross{5}{20};\ncross{6}{20};\ncross{7}{20};\ncross{8}{20};\ncross{-1}{20};
			\ncross{0}{21};\ncross{1}{21};\rcross{2}{21};\ncross{4}{21};\ncross{5}{21};\ncross{6}{21};\ncross{7}{21};\ncross{8}{21};\ncross{-1}{21};
			\ncross{0}{22};\ncross{1}{22};\ncross{2}{22};\ocross{3}{22};\ncross{5}{22};\ncross{6}{22};\ncross{7}{22};\ncross{8}{22};\ncross{-1}{22};
			\ncross{0}{23};\ncross{1}{23};\ocross{2}{23};\ncross{4}{23};\ncross{5}{23};\ncross{6}{23};\ncross{7}{23};\ncross{8}{23};\ncross{-1}{23};
			\ncross{0}{24};\ncross{1}{24};\ncross{2}{24};\ucross{3}{24};\ncross{5}{24};\ncross{6}{24};\ncross{7}{24};\ncross{8}{24};\ncross{-1}{24};
			\ncross{0}{25};\ncross{1}{25};\ucross{2}{25};\ncross{4}{25};\ncross{5}{25};\ncross{6}{25};\ncross{7}{25};\ncross{8}{25};\ncross{-1}{25};
			\ncross{0}{26};\ocross{1}{26};\ncross{3}{26};\ncross{4}{26};\ncross{5}{26};\ncross{6}{26};\ncross{7}{26};\ncross{8}{26};\ncross{-1}{26};
			\ncross{0}{27};\ncross{1}{27};\ncross{2}{27};\ncross{3}{27};\ucross{4}{27};\ncross{6}{27};\ncross{7}{27};\ncross{8}{27};\ncross{-1}{27};
			\ocross{0}{28};\ncross{2}{28};\ncross{3}{28};\ncross{4}{28};\ncross{5}{28};\ncross{6}{28};\ncross{7}{28};\ncross{8}{28};\ncross{-1}{28};
			
			\ncross{0}{29};\ncross{1}{29};\ncross{2}{29};\ncross{3}{29};\ncross{4}{29};\mcross{5}{3}{29};\ncross{8}{29};\ncross{-1}{29};
			\ncross{0}{30};\ncross{1}{30};\ncross{2}{30};\ncross{3}{30};\rcross{4}{30};\ncross{6}{30};\ncross{7}{30};\ncross{8}{30};\ncross{-1}{30};
			\ncross{0}{31};\ncross{1}{31};\ncross{2}{31};\rcross{3}{31};\ncross{5}{31};\ncross{6}{31};\ncross{7}{31};\ncross{8}{31};\ncross{-1}{31};
			\mcross{0}{4}{32};\ncross{4}{32};\ncross{5}{32};\ncross{6}{32};\ncross{7}{32};\ncross{8}{32};\ncross{-1}{32};
		\end{scope}
		
		\begin{scope}[shift={(7,2)}]
			\ncross{0}{30};\ncross{1}{30};\ncross{2}{30};\ncross{3}{30};\ocross{4}{30};\ncross{6}{30};\ncross{7}{30};\ncross{-1}{30};\ncross{-2}{30};
			\ncross{0}{31};\ncross{1}{31};\ncross{2}{31};\rcross{3}{31};\ncross{5}{31};\ncross{6}{31};\ncross{7}{31};\ncross{-1}{31};\ncross{-2}{31};
			\ncross{0}{32};\ncross{1}{32};\ncross{2}{32};\ocross{3}{32};\ncross{5}{32};\ncross{6}{32};\ncross{7}{32};\ncross{-1}{32};\ncross{-2}{32};
			\ncross{0}{33};\ncross{1}{33};\ocross{2}{33};\ncross{4}{33};\ncross{5}{33};\ncross{6}{33};\ncross{7}{33};\ncross{-1}{33};\ncross{-2}{33};
			\ncross{0}{34};\ncross{1}{34};\ncross{2}{34};\ncross{3}{34};\ncross{4}{34};\ocross{5}{34};\ncross{7}{34};\ncross{-1}{34};\ncross{-2}{34};
			\ncross{0}{35};\ncross{1}{35};\ncross{2}{35};\ncross{3}{35};\ocross{4}{35};\ncross{6}{35};\ncross{7}{35};\ncross{-1}{35};\ncross{-2}{35};
			\ncross{0}{36};\ncross{1}{36};\ncross{2}{36};\ocross{3}{36};\ncross{5}{36};\ncross{6}{36};\ncross{7}{36};\ncross{-1}{36};\ncross{-2}{36};
			\ncross{0}{37};\ncross{1}{37};\ncross{2}{37};\ncross{3}{37};\ncross{4}{37};\ucross{5}{37};\ncross{7}{37};\ncross{-1}{37};\ncross{-2}{37};
			
			\begin{scope}[shift={(1,0)}]
				\ncross{0}{37};\ncross{1}{37};\ncross{2}{37};\ncross{3}{37};\mcross{4}{4}{37};\ncross{-1}{37};\ncross{-2}{37};
				\ncross{0}{38};\ncross{1}{38};\mcross{2}{3}{38};\ncross{5}{38};\ncross{6}{38};\ncross{7}{38};\ncross{-1}{38};\ncross{-2}{38};
				\ncross{0}{39};\rcross{1}{39};\ncross{3}{39};\ncross{4}{39};\ncross{5}{39};\ncross{6}{39};\ncross{7}{39};\ncross{-1}{39};\ncross{-2}{39};
				\rcross{0}{40};\ncross{2}{40};\ncross{3}{40};\ncross{4}{40};\ncross{5}{40};\ncross{6}{40};\ncross{7}{40};\ncross{-1}{40};\ncross{-2}{40};
			\end{scope}
		\end{scope}
		
		\begin{scope}[shift={(10,3)}]
			\ncross{0}{39};\ncross{1}{39};\ucross{2}{39};\ncross{4}{39};\ncross{5}{39};\ncross{6}{39};\ncross{-1}{39};\ncross{-2}{39};\ncross{-3}{39};
			\ncross{0}{40};\ncross{1}{40};\ncross{2}{40};\ucross{3}{40};\ncross{5}{40};\ncross{6}{40};\ncross{-1}{40};\ncross{-2}{40};\ncross{-3}{40};
			\ncross{0}{41};\ncross{1}{41};\ucross{2}{41};\ncross{4}{41};\ncross{5}{41};\ncross{6}{41};\ncross{-1}{41};\ncross{-2}{41};\ncross{-3}{41};
			\ncross{0}{42};\rcross{1}{42};\ncross{3}{42};\ncross{4}{42};\ncross{5}{42};\ncross{6}{42};\ncross{-1}{42};\ncross{-2}{42};\ncross{-3}{42};
			\ncross{0}{43};\ncross{1}{43};\ncross{2}{43};\ocross{3}{43};\ncross{5}{43};\ncross{6}{43};\ncross{-1}{43};\ncross{-2}{43};\ncross{-3}{43};
			\ncross{0}{44};\ncross{1}{44};\ocross{2}{44};\ncross{4}{44};\ncross{5}{44};\ncross{6}{44};\ncross{-1}{44};\ncross{-2}{44};\ncross{-3}{44};
			\ncross{0}{45};\ncross{1}{45};\ncross{2}{45};\ncross{3}{45};\ncross{4}{45};\ncross{5}{45};\ncross{6}{45};\ncross{-1}{45};\ncross{-2}{45};\ncross{-3}{45};
	\end{scope}

		\node at (-1,9) {$L_{1}$};
		\node at (-1,8) {$L_{3}$};
		\node at (-1,7) {$L_{7}$};
		\node at (-1,6) {$L_{8}$};
		\node at (-1,5) {$L_{9}$};
		\node at (-1,4) {$L_{11}$};
		\node at (-1,3) {$L_{4}$};
		\node at (-1,2) {$L_{10}$};
		\node at (-1,1) {$L_{2}$};
		\node at (-1,0) {$L_{6}$};
		
	\end{scope}
\end{tikzpicture}
	\caption{Generic braided wiring diagram of $\mathcal{N}^+$ \label{Generic_WD_N+}}
\end{figure}

%%%%%%%%%%%%%%%%%%%%%%%%%%%%%%%%%%%%%%%%%%%%%%
\subsubsection{Method to compute the invariant}\mbox{}
%%%%%%%%%%%%%%%%%%%%%%%%%%%%%%%%%%%%%%%%%%%%%%

Let $\A=\set{D_1,\cdots,D_n}$ be a line arrangement, $\xi$ be a character and $\gamma_{(r,s,t)}$ be the cycle define by:
 \begin{equation*}
	\begin{tikzpicture}[xscale=1.5]
		\node (A1) at (0,0) {$v_{D_r}$};
		\node (A2) at (1,0) {$v_{P_{r,s}}$};
		\node (A3) at (2,0) {$v_{D_s}$};
		\node (A4) at (3,0) {$v_{P_{s,t}}$};
		\node (A5) at (4,0) {$v_{D_{t}}$};
		\node (A6) at (5,0) {$v_{P_{r,t}}$};
		\draw[->] (A1)  -- (A2);
		\draw[->] (A2) -- (A3);
		\draw[->] (A3) -- (A4);
		\draw[->] (A4) -- (A5);
		\draw[->] (A5) -- (A6);
		\draw[->] (A6) to[out=110,in=0] (3.5,0.75) -- (1.5,0.75) to[out=180,in=80] (A1);
	\end{tikzpicture}
\end{equation*}
We assume that $(\A,\xi,\gamma_{(r,s,t)})$ is an inner-cyclic arrangement. Let $W_\A$ be a wiring diagram of $\A$ such that the line $L_r$ is consider (in $W_\A$) as the line at infinity.

To compute the invariant $\mathcal{I}(\mathcal{A},\xi,\gamma_{(r,s,t)})$, we consider the singular braid formed by the part of $W_\A$ from the left of the diagram to the intersection point of $L_s$ and $L_t$ (excluding this point). Then, we construct a usual braid $\sigma^\A_{(r,s,t)}\in\BB_{n-1}$ by exchanging the singular points by a positive local half-twist as illustrated in Figure~\ref{half_twist}. 

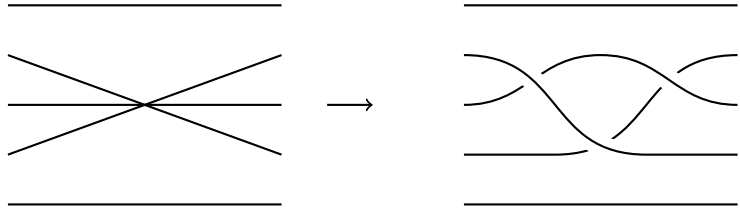
\begin{figure}[h!]
	\centering
	\begin{tikzpicture}
	\begin{scope}[xscale=0.6,yscale=0.33]
		\begin{scope}[shift={(-4,0)},xscale=1.5]
			\draw[thick] (-2,4) -- (2,4);
			\draw[thick] (-2,2) -- (2,-2);
			\draw[thick] (-2,0) -- (2,0);
			\draw[thick] (-2,-2) -- (2,2);
			\draw[thick] (-2,-4) -- (2,-4);
			
		\end{scope}
		
		\draw[thick,->] (0,0) -- (1,0);
		
		\begin{scope}[shift={(3,-2)}]
			\draw[thick] (0,6) -- (6,6);
			
			\draw[thick] (0,0) to[out=0,in=180] (2,0) to[out=0,in=180] (6,4);
			\draw[thick,color=white,line width=8pt] (0,2) to[out=0,in=180] (3,4) to[out=0,in=180] (6,2);
			\draw[thick] (0,2) to[out=0,in=180] (3,4) to[out=0,in=180] (6,2);
			\draw[thick,color=white,line width=8pt] (0,4) to[out=0,in=180] (4,0) to[out=0,in=180] (6,0);
			\draw[thick] (0,4) to[out=0,in=180] (4,0) to[out=0,in=180] (6,0);
			
			\draw[thick] (0,-2) -- (6,-2);
		\end{scope}
	\end{scope}
\end{tikzpicture}
	\caption{Construction of the braid $\sigma^\A_{(r,s,t)}$ \label{half_twist}}
\end{figure}

\begin{rmk}
	The braid $\sigma^\A_{(r,s,t)}$ is, in fact, the conjugating braid of any expression of the braid monodromy associated with $L_s\cap L_t$.
\end{rmk}

Finally~\cite[Proposition 4.3]{ArtFloGue} implies that the invariant is the image by $\xi$ of:
\begin{equation}\label{homological_inlcusion}
	\sum\limits_{i=0}\limits^{n} a_{i,s}(\sigma^\A_{(r,s,t)})\alpha_i - \sum\limits_{i=0}\limits^{n} a_{i,t}(\sigma^\A_{(r,s,t)})\alpha_i,
\end{equation}
where $a_{i,j}(\sigma^\A_{(r,s,t)})$ count with sign how many times the string associated with $D_i$ crosses over the string associated with $D_j$ in $\sigma^\A_{(r,s,t)}$. The braid $\sigma^\A_{(r,s,t)}$ is oriented from the left to the right, and the sign of the crossing is illustrated in Figure~\ref{Crossing_Sign}. For more details on the computation of the invariant see~\cite[Section 4]{ArtFloGue}.

\begin{figure}[h!]
	\centering
	\begin{tikzpicture}
	\begin{scope}
		
		\begin{scope}[shift={(-3,0)}]
			\draw[thick,->] (-1,0) -- (1,2);
			\draw[thick,->,line width=10pt, color=white] (-1,2) -- (1,0);
			\draw[thick,->] (-1,2) -- (1,0);
			\node at (0,-1) {\small{(\textsc{a}) Positive crossing}};
		\end{scope}
				
		\begin{scope}[shift={(3,0)}]
			\draw[thick,->] (-1,2) -- (1,0);		
			\draw[thick,->,line width=10pt, color=white] (-1,0) -- (1,2);
			\draw[thick,->] (-1,0) -- (1,2);
			\node at (0,-1) {\small{(\textsc{b}) Negative crossing}};
		\end{scope}
		
	\end{scope}
\end{tikzpicture}
	\caption{ Sign of the crossing in $\sigma^\A_{(r,s,t)}$ \label{Crossing_Sign}}
\end{figure}
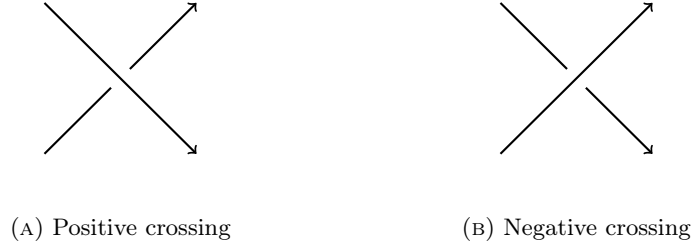

\begin{rmk}
	By convention, for all $i$, $a_{i,i}=0$.
\end{rmk}

\begin{rmk}
	Since $D_r$ is a fiber of $p$, then no line can over crossing it. Then, we do not need consider this line in the previous computation, and only what's happen to $D_s$ and $D_t$ determine the value of the invariant.
\end{rmk}

%%%%%%%%%%%%%%%%%%%%%%%%%%%%%%%%%%%%%%%%%%%%%%
\subsubsection{Computation for $\mathcal{N}^+$ and $\mathcal{M}^+$}\mbox{}
%%%%%%%%%%%%%%%%%%%%%%%%%%%%%%%%%%%%%%%%%%%%%%

To apply the previous algorithm at $\mathcal{N}^+$ and $\mathcal{M}^+$, we take $D_r=L_5$, $D_s=L_6$ and $D_t=L_{11}$; and the character $\xi$ considered is:
\begin{equation*}
		\xi:(L_1,\cdots,L_{11}) \longmapsto (\zeta,\zeta^4,\zeta^3,\zeta^2,1,1,\zeta,\zeta^2,\zeta^3,\zeta^4,1),
\end{equation*}
By Proposition~\ref{Propo_N_M_inner_cylic}, $(\mathcal{N}^+,\xi,\gamma_{(5,6,11)})$ and $(\mathcal{M}^+,\xi,\gamma_{(5,6,11)})$ are inner-cylic arrangements. In the following, we completely detailed the computation of $\mathcal{I}(\mathcal{N}^+,\xi,\gamma_{(5,6,11)})$.

$\bullet$ For $\mathcal{N}^+$: The braid $\sigma^{\mathcal{N}^+}_{(5,6,11)}$ obtain from its generic braided wiring diagram (see  Figure~\ref{Generic_WD_N+}), is pictured in Figure~\ref{braid_N}.
\begin{figure}[h!]
	\centering
	\begin{tikzpicture}
	\begin{scope}[xscale=0.45,yscale=0.5]
	
	\node at (-1,1) {$L_6$};
	\node at (-1,2) {$L_2$};
	\node at (-1,3) {$L_{10}$};
	\node at (-1,4) {$L_4$};
	\node at (-1,5) {$L_{11}$};
	\node at (-1,6) {$L_9$};
	\node at (-1,7) {$L_8$};
	\node at (-1,8) {$L_7$};
	\node at (-1,9) {$L_3$};
	\node at (-1,10) {$L_1$};
	
	 %début
		\nc{0}{10};
		
		\oc{5}{1}{10};
		\oc{4}{2}{10};
		
		%double
		\oc{3}{3}{10};
		
		\oc{3}{4}{10};
		\uc{4}{5}{10};
		\uc{5}{6}{10};
		
		%triple
		\oc{5}{7}{10};
		\oc{4}{8}{10};
		\oc{5}{9}{10};
		
		\oc{4}{10}{10};
		\oc{5}{11}{10};
		\uc{1}{12}{10};
		\uc{3}{13}{10};
		\uc{2}{14}{10};
		
		%verticale
		\oc{7}{15}{10};
		\oc{6}{16}{10};
		\oc{5}{17}{10};
		%triple
		\oc{4}{18}{10};
		\oc{3}{19}{10};
		\oc{4}{20}{10};
		%triple
		\oc{2}{21}{10};
		\oc{1}{22}{10};
		\oc{2}{23}{10};
		
		\uc{2}{24}{10};
		\uc{6}{25}{10};
		\uc{4}{26}{10};
		\oc{3}{27}{10};
		\uc{5}{28}{10};
		\uc{4}{29}{10};
		\uc{3}{30}{10};
		
		% fin 
		\nc{31}{10};
		
		% cercles
		\draw[thick,dotted] (3.5,3.5) circle (0.66);
		\draw[thick,dotted] (8.5,5) circle (1.5);
		\draw[thick,dotted] (15.5,7.5) circle (0.66);
		\draw[thick,dotted] (16.5,6.5) circle (0.66);
		\draw[thick,dotted] (17.5,5.5) circle (0.66);
		\draw[thick,dotted] (19.5,4) circle (1.5);
		\draw[thick,dotted] (22.5,2) circle (1.5);
		
	\end{scope}
\end{tikzpicture}
	\caption{The braid $\sigma^{\mathcal{N}^+}_{(5,6,11)}$ \label{braid_N}}
\end{figure}
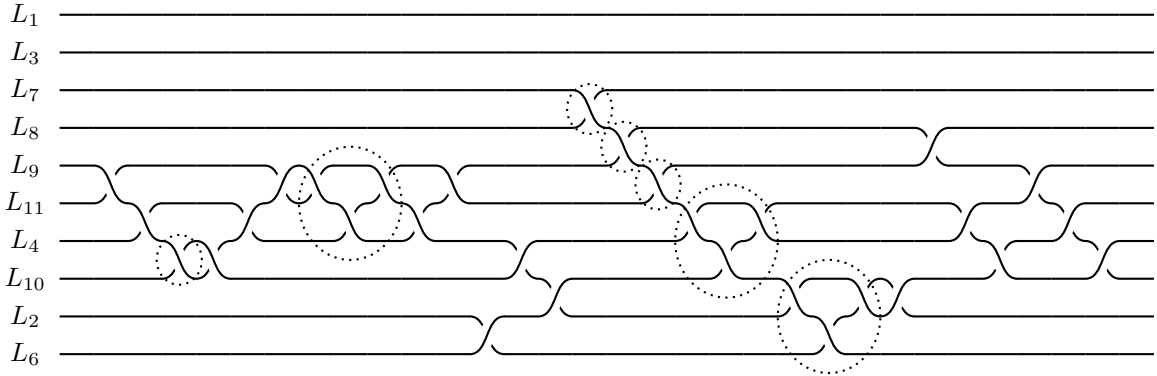

\begin{rmk}
	The circled crossings come from the exchange of the singular points of $W_{\mathcal{N}^+}$ into local half-twists.
\end{rmk}

To determine the value of~(\ref{homological_inlcusion}) we proceed in a two-fold manner. Firstly, we add (with sign) the meridians of the lines over crossing the wire 6. Secondly, we subtract (with sign) the meridians of the lines over crossing the wire 11.

The wire 6 is over crossed:
\begin{itemize}
	\item[-] two times by the wire 10 (one positively and one negatively),
	\item[-] one time positively by the wire 7.
\end{itemize}

The wire 11 is over crossed:
\begin{itemize}
	\item[-] three times by the wire 9 (twice positively and one negatively),
	\item[-] one time negatively by the wire 10,
	\item[-] two times by the wire 6 (one positively and one negatively),
	\item[-] one time positively by the wire 7.
\end{itemize}

Thus, we obtain that:
\begin{equation*}
	\mathcal{I}(\mathcal{N}^+,\xi,\gamma_{(5,6,11)})=\xi\big( (\alpha_7) - (\alpha_9-\alpha_{10}+\alpha_7)\big ) = \zeta.
\end{equation*}

$\bullet$ For $\mathcal{M}^+$: From a perturbation of the non-generic braided wiring diagram pictured in Figure~\ref{11_prime_B}, we construct the braid $\sigma^{\mathcal{M}^+}_{(5,6,11)}$. It is pictured in Figure~\ref{braid_M}.

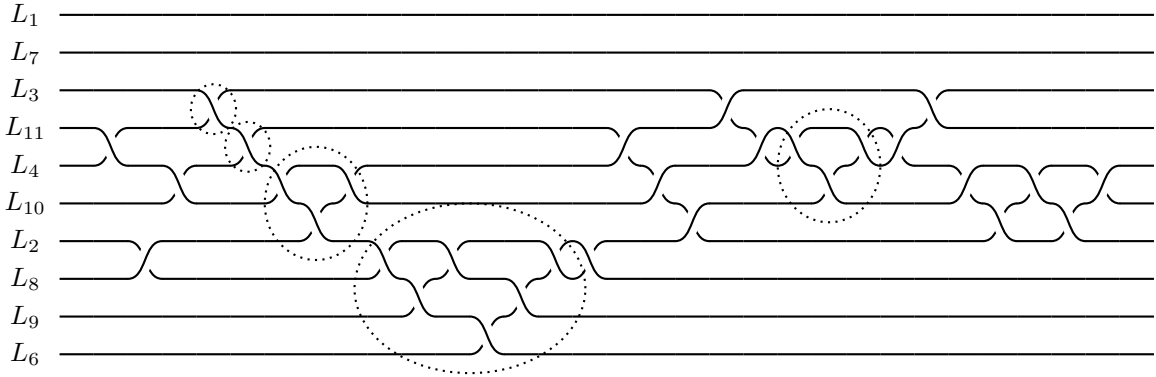
\begin{figure}[h!]
	\centering
	\begin{tikzpicture}
	\begin{scope}[xscale=0.45,yscale=0.5]
	
		\node at (-1,1) {$L_6$};
		\node at (-1,2) {$L_9$};
		\node at (-1,3) {$L_8$};
		\node at (-1,4) {$L_2$};
		\node at (-1,5) {$L_{10}$};
		\node at (-1,6) {$L_4$};
		\node at (-1,7) {$L_{11}$};
		\node at (-1,8) {$L_3$};
		\node at (-1,9) {$L_7$};
		\node at (-1,10) {$L_1$};
		
		% début
		\nc{0}{10};
		\oc{6}{1}{10};
		\uc{3}{2}{10};
		\oc{5}{3}{10};
		
		%vertical
		\oc{7}{4}{10};
		\oc{6}{5}{10};
		%triple
		\oc{5}{6}{10};
		\oc{4}{7}{10};
		\oc{5}{8}{10};
		%quad
		\oc{3}{9}{10};
		\oc{2}{10}{10};
		\oc{3}{11}{10};
		\oc{1}{12}{10};
		\oc{2}{13}{10};
		\oc{3}{14}{10};
		
		\oc{3}{15}{10};
		\uc{6}{16}{10};
		\uc{5}{17}{10};
		\uc{4}{18}{10};
		\uc{7}{19}{10};
		\uc{6}{20}{10};
		
		%triple
		\oc{6}{21}{10};
		\oc{5}{22}{10};
		\oc{6}{23}{10};
		
		\uc{6}{24}{10};
		\oc{7}{25}{10};
		\uc{5}{26}{10};
		\oc{4}{27}{10};
		\oc{5}{28}{10};
		\oc{4}{29}{10};
		\uc{5}{30}{10};
		
		%fin
		\nc{31}{10};

		% cercles
		\draw[thick,dotted] (4.5,7.5) circle (0.66);
		\draw[thick,dotted] (5.5,6.5) circle (0.66);
		\draw[thick,dotted] (7.5,5) circle (1.5);
		\draw[thick,dotted,xscale=1.5] (8,2.75) circle (2.25);
		\draw[thick,dotted] (22.5,6) circle (1.5);
		
	\end{scope}
\end{tikzpicture}
	\caption{The braid $\sigma^{\mathcal{M}^+}_{(5,6,11)}$ \label{braid_M}}
\end{figure}

\begin{rmk}\mbox{}
	\begin{enumerate}
		\item The perturbation applied to obtain the generic braided wiring diagram from the non-generic one pictured in Figure~\ref{11_prime_B} is such that the three vertical lines have very negative slope.  
		\item The circled crossings come from the exchange of the singular points of $W_{\mathcal{M}^+}$ into local half-twists.
	\end{enumerate}
\end{rmk}

After computation, we obtain that:
\begin{equation*}
	\mathcal{I}(\mathcal{M}^+,\xi,\gamma_{(5,6,11)})=\xi\big( (\alpha_3+\alpha_9+\alpha_2) - (\alpha_3) \big)=\zeta^2.
\end{equation*}

By~\cite[Proposition 2.5]{ArtFloGue} we know that the invariant and the complex conjugation commute, then:
\begin{equation*}
	\mathcal{I}(\mathcal{N}^-,\xi, \gamma_{(5,6,11)})=\overline{\zeta}=\zeta^4 \quad \text{and} \quad \mathcal{I}(\mathcal{M}^-,\xi, \gamma_{(5,6,11)})=\overline{\zeta^2}=\zeta^3.
\end{equation*}

By Theorem~\ref{Thm_TopologicalInvariant}, we have proved Theorem~\ref{Thm_OrientedOrderedPair}. Thus, to delete the ``orientation-preserving'' condition of Theorem~\ref{Thm_TopologicalInvariant}, we consider the following lemma. 
\begin{lem}
	Let $\A_1$ and $\A_2$ be two arrangements	with the same combinatorics and such that there is no homeomorphism preserving both orientation and order between $(\CC\PP^2,\A_1)$ and $(\CC\PP^2,\A_2)$. If there is no orientation-preserving homeomorphism between $\A_2$ and the complex conjugate of $\A_1$ then there is no order-preserving homeomorphism between $(\CC\PP^2,\A_1)$ and $(\CC\PP^2,\A_2)$.
\end{lem}
It is a consequence of~\cite[Theorem~4.19]{ArtCarCogMar_Topology} (see also~\cite[Theorem 6.4.8]{Gue_Thesis} for a complete proof). Applying this lemma at Theorem~\ref{Thm_OrientedOrderedPair} (\emph{i.e.} to the pairs $(\CC\PP^2,\mathcal{N}^\pm)$ and $(\CC\PP^2,\mathcal{M}^\pm)$), we obtain Corollary~\ref{Cor_OrderedPair}.

\section*{Acknowledgments}
This result is an improvement of a result obtained during my PhD supervised by E.~\textsc{Artal} and V.~\textsc{Florens}. Thanks to them for their supervisions. The author would also like to thank M.A.~\textsc{Marco Buzun{\'a}riz} for his list of combinatorics which was the starting point for this research into \textsc{Zariski} pair.

\bibliographystyle{amsplain}
\bibliography{biblio}

\end{document}